\documentclass[a4paper,12pt,reqno]{amsart}
\usepackage{amsmath,amsthm,amssymb}
\usepackage{stmaryrd}
\usepackage{url}
\usepackage{csquotes}
\usepackage[left=0.9in, right=0.9in,top=0.9in, bottom=1in,includehead]{geometry}
\usepackage{xcolor}
\usepackage{parskip}
\usepackage{mathtools}
\mathtoolsset{showonlyrefs}
\usepackage[activate={true,nocompatibility},final,tracking=true,kerning=true,spacing=true]{microtype}
\usepackage[pdfencoding=unicode,pdftex,bookmarks=true,bookmarksnumbered]{hyperref}

\definecolor{citecolour}{rgb}{0.0, 0.0, 0.8}
\colorlet{linkcolour}{green!50!black}
\hypersetup{colorlinks,breaklinks,
		linkcolor=linkcolour,citecolor=citecolour,
		filecolor=linkcolour, urlcolor=linkcolour}

\numberwithin{equation}{section}

\theoremstyle{plain}
\newtheorem*{theorem*}{Theorem}
\newtheorem{theorem}{Theorem}[section]
\newtheorem{lemma}[theorem]{Lemma}
\newtheorem{corollary}[theorem]{Corollary}
\newtheorem*{theoremA}{Theorem A}
\newtheorem*{theoremB}{Theorem B}

\theoremstyle{definition}

\newtheorem{remark}[theorem]{Remark}

\newtheorem{example}[theorem]{Example}

\newtheorem*{ack}{Acknowledgements}

\usepackage[shortlabels]{enumitem}
\begingroup
    \makeatletter
    \@for\theoremstyle:=definition,remark,plain\do{%
        \expandafter\g@addto@macro\csname th@\theoremstyle\endcsname{%
            \addtolength\thm@preskip\parskip
            }%
        }
\endgroup

\DeclareMathOperator{\F}{F}
\DeclareMathOperator{\Z}{Z}
\DeclareMathOperator{\Soc}{Soc}
\DeclareMathOperator{\Aut}{Aut}

\DeclareMathOperator{\Lay}{E}
\DeclareMathOperator{\Fast}{F^{\ast}}
\DeclareMathOperator{\Fprime}{F^{\prime}}

\DeclareMathOperator{\nc}{\textbf{nC}}

\DeclareMathOperator{\B}{\mathfrak{B}}
\DeclareMathOperator{\FF}{\mathfrak{F}}

\renewcommand{\leq}{\leqslant}
\renewcommand{\geq}{\geqslant}

\begin{document}
\title{On the Frattini subgroup of a finite group}
\author{Stefanos Aivazidis}
\address{Departament d' \'Algebra, Universitat de Val\`encia, C/ Dr. Moliner, 50 46100-Burjassot (Val\`encia), Spain}
\email{s.aivazidis@qmul.ac.uk}
\author{Adolfo Ballester-Bolinches}
\address{Departament d' \'Algebra, Universitat de Val\`encia, C/ Dr. Moliner, 50 46100-Burjassot (Val\`encia), Spain}
\email{adolfo.ballester@uv.es}
\thanks{The first author has been supported by an LMS 150th Anniversary Postdoctoral Mobility Grant. The second author has been supported by the grant MTM2014-54707-C3-1-P from the Ministerio de Econom{\'\i}a y Competitividad, Spain, and FEDER, European Union. He has also been supported by a project from the National Natural Science Foundation of China (NSFC,
No.~11271085) and a project of Natural Science Foundation of Guangdong  Province (No.~2015A030313791).}
\begin{abstract}
We study the class of finite groups $G$ satisfying $\Phi (G/N)= \Phi(G)N/N$ for all normal subgroups $N$ of $G$. As a consequence of our main results we extend and amplify a theorem of Doerk concerning this class from the soluble universe to all finite groups and answer
in the affirmative a long-standing question of Christensen whether the class of finite groups which possess complements for each of their normal subgroups is subnormally closed.
\end{abstract}
\maketitle

\section{Introduction and statement of results}
The only groups considered in this paper are finite. In the present article we shall examine certain questions concerning the behaviour of the Frattini subgroup in epimorphic images.

Following Gasch{\"u}tz~\cite{gasch}, we call a group $\Phi$-free if its Frattini subgroup is trivial. We denote by $\B$ the class of groups $G$ such that $G/N$ is $\Phi$-free for all normal subgroups $N$ of $G$. It is clear that a group $G \in \B$ if and only if $G$ has no Frattini chief factors (recall that a chief factor $K/L$ of a group $G$ is said to be Frattini if $K/L \leq \Phi(G/L)$).

Our first main result is a reduction theorem of wide applicability that provides a sufficient condition for a $\Phi$-free group to belong to $\B$.

\begin{theoremA}
Suppose that $G$ is a $\Phi$-free group. If $G/N$ is $\Phi$-free for all normal subgroups $N$ of $G$ containing the generalised Fitting subgroup $\Fast(G)$ then $G$ belongs to $\B$.
\end{theoremA}

Theorem A is important in the study of groups that behave 
like nilpotent groups with respect to the Frattini subgroup, that is, $\Phi(G/N) = \Phi(G)N/N$ for all $N \mathrel{\lhd} G$. In one of his last papers, Doerk~\cite{doerk} examined the soluble case and proved that the class $\mathfrak{F}_{\text{sol}}$ of soluble groups $G$ where 
$\Phi (G/N)= \Phi(G)N/N$ for every normal subgroup $N$ of $G$ is a saturated formation, that is, a class of groups which is closed under epimorphic images, subdirect products and Frattini extensions.  Further, he obtained several equivalent conditions for a soluble group to be in $\mathfrak{F}_{\text{sol}}$.

\begin{theorem}[{\cite[Satz 2']{doerk}}]\label{satz2'}
Let $G$ be a soluble group. Then the following statements are pairwise equivalent:
\begin{enumerate}[label={\upshape(\arabic*)}]
\item $G \in \mathfrak{F}_{\text{sol}}$.
\item $G /\Phi(G)$ has no Frattini chief factors.
\item $G /\F(G)$ has no Frattini chief factors.
\item If $H/K$ is a chief factor of $G$ then $G /C_G(H/K)$ has no Frattini chief factors.
\end{enumerate}
\end{theorem}

One well-known feature of a saturated formation $\FF$ is that in each group $G$, every chief factor of the form $G^{\FF}/K$ is supplemented in $G$, where $G^{\FF}$ is the $\FF$-residual of $G$, that is, the smallest normal subgroup of $G$ with quotient in $\FF$. \emph{Totally non-saturated formations} (or tn-formations for short) are studied in~\cite{doerkh},~\cite{locom} in the soluble universe, and in~\cite{prefrattini2},~\cite{bjuan} in the general case, as 
the formations $\FF$ such that, in each group $G$, every chief factor of the form $G^{\FF}/K$ is Frattini. 

Totally non-saturated formations are significant in many ways, one of which is the fact that every soluble group can be embedded in a  multiprimitive group (\cite{hawkes}), and these groups belong to every totally non-saturated formation of full characteristic.

Our second main result is a generalisation of Doerk's theorem, valid for all
groups. Its proof depends heavily on Theorem~A.

\begin{theoremB}
The following assertions are valid:
\begin{enumerate}[label={\upshape(\arabic*)}]
\item\label{thmB1} The class $\B$ is a subnormally closed tn-formation.

\item\label{thmB2} The class $\FF$ of all groups $G$ such that $\Phi (G/N)= \Phi(G)N/N$ for all $N \mathrel{\lhd} G$ satisfies $\mathfrak{F} = E_{\Phi}\B = \mathfrak{N}\B$, where $\mathfrak{N}$ is the class of all nilpotent groups. 

\item\label{thmB3} The class $\FF$ is the smallest saturated formation containing $\B$ and it is locally defined by the formation function $f$ given by $f(p) = \B$ for all primes $p$. Moreover, $\FF$ is closed under taking subnormal subgroups. 
\end{enumerate}
\end{theoremB}

\begin{corollary}\label{newdoerk}
Let $G$ be a group. Then the following statements are pairwise equivalent:
\begin{enumerate}[label={\upshape(\arabic*)}]
\item\label{newdoerk1} If $N \mathrel{\lhd} G$ then $\Phi (G/N) = \Phi(G)N/N$.
\item\label{newdoerk2} $G /\Phi(G)$ has no Frattini chief factors.
\item\label{newdoerk3} $G/\F(G)$ has no Frattini chief factors. 
\item\label{newdoerk4} $G /\Fprime(G)$ has no Frattini chief factors.
\item\label{newdoerk5} If $H/K$ is a chief factor of $G$, then $G /C_G(H/K)$ has no Frattini chief factors.
\end{enumerate}
\end{corollary}

Following Christensen~\cite{christensen1}, we say that $G$ is 
an $\nc$-group if every normal subgroup of $G$ is complemented. The class of all $\nc$-groups is a tn-formation which is contained in every tn-formation of full characteristic. In particular, every $\nc$-group is a $\B$-group. However, the containment of $\nc$ in $\B$ is proper. An example of a $\B$-group which has a normal subgroup with no complement is $\Aut(A_6)$.

\begin{theorem}\label{th-1}
Suppose that $G$ is a $\B$-group. Then $G$ is an $\nc$-group if and only if $G$ splits over each member of its generalised Fitting series.
\end{theorem}

The following known result is a direct consequence of Theorem~\ref{th-1}.
\begin{corollary}[{\cite{bechtell-1}}]
Every soluble group is a $\B$-group if and only if it is an $\nc$-group.
\end{corollary} 

Theorem~\ref{th-1} also serves to answer in the affirmative a long standing question whether a subnormal subgroup of an $\nc$-group is also an $\nc$-group (see~\cite{christensen1}). 
\begin{corollary}\label{subnc}
The class $\nc$ is subnormally closed.
\end{corollary}

\section{Notation and Preliminaries}
It is assumed that the reader is familiar with the notation presented 
in~\cite{classes} and~\cite{dh}. In order to make our paper reasonably self-contained, we collect in the following lemma some well-known facts about the Frattini subgroup of a group. Most of these will be used without further explicit reference. For full proofs the reader should consult Gasch{\"u}tz's early paper~\cite{gasch} or~\cite[pp.~30--32]{dh}.
\begin{lemma}\label{collection}
Let $G$ be a group.
\begin{enumerate}[label={\upshape(\roman*)}]
\item Let $N\mathrel{\lhd}G$. Then $N$ has a proper supplement in $G$ if and only if $N$ is not contained in $\Phi(G)$. And if $U$ is minimal in the set of supplements for $N$ in $G$ then $N \cap U = N \cap \Phi(U)$.
\item If $G = H\times K$ then $\Phi(G) = \Phi(H)\times\Phi(K)$.
\item If $N\mathrel{\lhd}G$ then $\Phi(G)N/N \leq \Phi(G/N)$, and if
$N \leq\Phi(G)$ then $\Phi(G)/N =\Phi(G/N)$.
\item If $L \leq G$ and $U \leq \Phi(L)$ with $U\mathrel{\lhd}G$ then $U \leq \Phi(G)$; thus if $L\mathrel{\lhd}G$ then $\Phi(L) \leq \Phi(G)$.
\item If $G$ is a $p$-group then $\Phi(G) = G^{\prime}G^p$, thus $G/\Phi(G)$ is elementary abelian. Therefore, if $G$ is nilpotent and $N\mathrel{\lhd}G$ then $\Phi(G/N)=\Phi(G)N/N$.
\item\label{satz4}$\Phi(G) \geq \Z(G) \cap G^{\prime}$.
\item\label{satz7} If $A$ is an abelian normal subgroup of $G$ and $A \cap \Phi(G)=1$ then $A$ is complemented in $G$.
\end{enumerate}
\end{lemma}
Let us also fix some further notation and terminology. We say that a group $G$ is \emph{quasi-simple} if $G$ is perfect and $G/\Z(G)$ is simple.
A subnormal quasi-simple subgroup of an arbitrary group $G$ is called a \emph{component} of $G$ and the \emph{layer}
of $G$, denoted by $\Lay(G)$, is the product of its components. It is known that if $A$ and $B$ are components of $G$, then either $A = B$ or $[A, B] = 1$ (see~\cite[Section 2.2]{classes}).

The \emph{generalised Fitting subgroup} of $G$, denoted by $\Fast(G)$,
is the product of $\Lay(G)$ and $\F(G)$, the Fitting subgroup of $G$. 
The associated generalised Fitting series of $G$ is defined by induction 
as $\F_0^{\ast} \coloneqq 1$, and 
$\F_{i+1}^{\ast}/\F_i^{\ast} \coloneqq \Fast(G/\F_i^{\ast})$ for $i>0$. 

For any group $G$, $\Fprime(G)$ denotes the normal subgroup of $G$ such that $\Fprime(G)/\Phi(G) = \Soc(G/\Phi(G))$. Note that $\F(G)$ is a subgroup of $\Fprime(G)$, and if $G$ is soluble then $\F(G) = \Fprime(G)$. Finally, if $H,K$ are subgroups of a group $G$ then we shall write $\llbracket H, K\rrbracket \coloneqq \left\{ L \leq G : H \leq L \leq K \right\}$.

The proof of Theorem~A depends on the following property of the generalised Fitting subgroup. 

\begin{lemma}\label{aux}
Suppose that $G$ is a group with $\Phi(G)=1$. Then $\Fprime(G) = \Fast(G)=\F(G)\times\Lay(G)$, and if $N \mathrel{\lhd} G$ with $N \leq \Fast(G)$ then $N=\F(N)\times \Lay(N)$. Also, if $V$ is another complement for $\F(N)$ in $N$ then $V=\Lay(N)$.
\end{lemma}
\begin{proof}
First, we argue that $\Fast(G)=\F(G)\times\Lay(G)$ follows from $\F(G) \cap \Lay(G)=\F(\Lay(G))=1$ 
and we justify the latter. If $Q$ is a component of $G$, then $Q$ is subnormal in $G$ so 
$\Phi(Q) \leq \Phi(G)=1$ and thus $\Phi(Q)=1$. As $Q/\Z(Q)$ is non-abelian simple, it follows that 
$\Phi(Q) \leq \Z(Q)$. On the other hand, we have $\Z(Q) = Q^{\prime} \cap \Z(Q) \leq \Phi(Q)$ from 
Lemma~\ref{collection}~\ref{satz4} and the fact that $Q$ is perfect. Therefore $\Z(Q)=\Phi(Q)=1$ and so $Q$ is 
non-abelian simple. Thus $\Lay(G)$ is a direct product of non-abelian simple groups and so $\F(\Lay(G))=1$, as wanted. By 
\cite[Lemma 4.1]{dh}, $\Lay(G)$ is contained in $\Soc(G) = \Fprime(G)$. Therefore, $\Fprime(G) = \Fast(G)$.

Let $N$ be a normal subgroup of $G$ contained in $\Fast(G)$. Then $N = \Fast(N) = \F(N)\times\Lay(N)$ since $\Phi(N)=1$. Hence $\Lay(N)$ is a direct product of components of $G$. By \cite[Lemma 4.1]{dh}, we have that $\Lay(N) = N\cap\Lay(G)$. Thus $N=(N\cap\F(G))\times(N\cap\Lay(G))=\F(N)\times\Lay(N)$. Since $\F(N)$ is abelian, it follows that $\F(N)=\Z(N)$.

Suppose that $V$ is another complement for $\F(N)$ in $N$. Then $N=\F(N)V=\F(N) \times V$. Thus $V=N^{\prime}=\Lay(N)$, as desired.
\end{proof}
\section{Good normal subgroups}\label{good}

We begin by identifying certain normal subgroups $N$ of a group $G$ which satisfy the equation $\Phi(G/N)=\Phi(G)N/N$, even if $G$ 
as a whole does not have this property (an example that easily comes to mind is the Frobenius group of order 20). We refer to these subgroups as ``good" normal subgroups.

\begin{lemma}\label{frattavoid}
Let $N$ be a $\Phi$-free central subgroup of the group $G$. Then $N$ is a good normal subgroup of $G$. 
\end{lemma}
\begin{proof}
Let $E$ be a not necessarily proper or nontrivial complement to $N \cap \Phi(G)$ in $N$ so that 
$N = \left(N \cap \Phi(G)\right) \times E$. The existence of $E$ is guaranteed since $N$, being central 
and $\Phi$-free, is a direct product of elementary abelian subgroups and such a group is complemented~\cite{hall}. 
Then $E \cap \Phi(G) = 1$, hence $E$ has a complement in $G$, say $L$. So $G = E \times L$ and $\Phi(G) = \Phi(L)$. Since $G/N$ is isomorphic to $L\big/N \cap \Phi(G)$, it follows that 
\begin{equation}
\Phi(G/N) \cong \Phi \left( L\big/N \cap \Phi(L) \right) = \Phi(L)/N \cap \Phi(L) \cong \Phi(L)N\big/N = \Phi(G)N\big/N.
\end{equation}
Consequently, $\Phi(G/N) = \Phi(G)N\big/N$ and the claim follows.
\end{proof}

In particular, if $N \leq \Soc(Z(G))$ then $\Phi(G/N)=\Phi(G)N/N$. In fact, we can use this information to get more good normal subgroups.
\begin{corollary}\label{soclecentre}
Let $G$ be a group and define by induction $S_0 \coloneqq 1$, and 
$S_{i+1}/S_i \coloneqq \Soc(Z(G/S_i))$ for $i>0$. 
Then every normal subgroup of $G$ lying in an interval 
$\llbracket S_j , S_{j+1} \rrbracket$, $j \geq 0$ is good.
\end{corollary}
\begin{proof}
We induct on $j$. For $j=0$ the claim is true from Lemma~\ref{frattavoid}. Now, let $H \mathrel{\lhd}G$ with 
$H \in \llbracket S_j , S_{j+1} \rrbracket$, $j>0$: Write $\overline{G} = G/S_j$, and use the bar convention. Then 
$\overline{H}$ is a subgroup of $\Soc(Z(\overline{G}))$, so by Lemma~\ref{frattavoid} we have
$\Phi(\overline{G}/\overline{H})=\Phi(\overline{G})\overline{H}\big/\overline{H}$. Note that 
$\Phi(\overline{G}/\overline{H}) \cong \Phi(G/H)$, so it suffices to show that 
$\Phi(\overline{G})\overline{H}\big/\overline{H} \cong \Phi(G)H/H$ to complete the proof. By induction 
$\Phi(\overline{G})=\Phi(G)S_j/S_j$, thus $\Phi(\overline{G})\overline{H}=\left(\Phi(G)S_j/S_j\right)H/S_j=\Phi(G)H/S_j$. Then
clearly $\Phi(\overline{G})\overline{H}\big/\overline{H} \cong \Phi(G)H/H$ and the induction is complete.
\end{proof}
Note that if $M,N$ are good normal subgroups of $G$, then $MN$ is not necessarily good. 
\begin{example}
Let $E$ be elementary abelian of order $5^2$ and assume that $C = \langle c \rangle$ is cyclic of order 4, acting on $E$ 
by $x^c = x^2$, $x \in E$. Let $G = E \rtimes C$, the semidirect product, 
and write $E = M \times N$, where each of $M$ and $N$ has order 5. Note that $M$ and $N$ are normal
in $G$. Also $\Phi(G) = 1$. To see this, note that $MC$ and $NC$ are maximal in $G$, so $\Phi(G) \leq MC \cap NC = C$. 
Then $\Phi(G) \cap E =1$ so $\Phi(G)$ centralises $E$. Since $E$ is self-centralising in $G$, this forces $\Phi(G) = 1$.
Now $K \mathrel{\lhd} G$ is good if and only if $\Phi(G/K)$ is trivial. Then $M$ and $N$ are good 
because $G/M$ and $G/N$ are each isomorphic to the Frobenius group of order 20 and so have trivial Frattini subgroups. But $MN = E$ is a normal subgroup of $G$ which is not good. 
\end{example}
\begin{remark}
There is an obvious exception to the general rule that if $M,N$ are normal good subgroups of a group $G$, then the product $MN$ is not good. 
If $M,N$ are normal in $G$, both good, and both lie in the same interval $\llbracket S_j , S_{j+1} \rrbracket$ for some index $j$, where $S_j, S_{j+1}$ are successive terms of the $S$-series of Corollary~\ref{soclecentre}, then their product is normal in $G$ and lies in $\llbracket S_j , S_{j+1} \rrbracket$. Therefore, in this particular case, $MN$ is good.
\end{remark}
\section{Proofs of the main results}

\begin{proof}[Proof of Theorem~A]
Suppose that $G/N$ is $\Phi$-free for all normal subgroups $N$ of $G$ containing $\Fast(G)$. Assume, arguing by contradiction, that the result is not true. Then $G$ has a normal subgroup $E$ such that $G/E$ is not $\Phi$-free. Let us choose $E$ of minimal order. If $1<N<E$ with $N\mathrel{\lhd} G$, we will show that $G/N$ satisfies the hypotheses of the theorem. First, $\Phi(G/N)=1$ by the choice of $E$. Next, if $\Fast(G/N) \leq M/N$ with $M \mathrel{\lhd} G$ then $\Fast(G)\leq M$, so by hypothesis $\Phi(G/M)=1$, thus $\Phi(G/N \big/ M/N)=1$. The minimal choice of $G$ implies that $\Phi\left(G/N \big/ E/N\right)=1$, so $\Phi(G/E)=1$, contrary to supposition. We may thus assume that $E$ is minimal normal in $G$. Write $U/E =\Phi(G/E)$. Any maximal subgroup of $G$ which contains $\Fast(G)$ contains $E$ and thus contains $U$ and so contains $U\Fast(G)$. Therefore $U\Fast(G)/U$ is a subgroup of $\Phi(G/\Fast(G))=1$, thus $U \leq \Fast(G)$. By Lemma~\ref{aux} we deduce that $U =\F(U) \times \Lay(U)$. 

Assume that $E$ is abelian. Then $U$ is soluble, $\Lay(U)=1$ and $U=\F(U)$ is nilpotent. Also $U \mathrel{\lhd} G$ so $\Phi(U) \leq \Phi(G)=1$, hence $\Phi(U)=1$ and thus $U$ is abelian. Then Lemma~\ref{collection}~\ref{satz7} guarantees the existence of a complement for $U$ in $G$. Call this complement $X$. Then $XE/E$ is a complement for $U/E$ in $G/E$ and since $U/E =\Phi(G/E)$ we have $G = XE$. By Dedekind's lemma $U=U \cap XE=(U \cap X)E=E$, thus $\Phi(G/E)=1$, contrary to our assumption. 

Assume that $E$ is non-abelian. Then $E\leq\Lay(U)$. In fact, $E$ must be the whole of $\Lay(U)$ for if not then $U/\F(U)E$, which is a homomorphic image of the nilpotent $U/E=\Phi(G/E)$ thus nilpotent itself, would also be a direct product of non-abelian simple groups. 
Therefore $U=\F(U)\times E$. By Lemma~\ref{collection}~\ref{satz7}, and since $\F(U)$ is
abelian, there is a complement $X$ for $\F(U)$ in $G$. Let $V=U\cap X$. Then Dedekind's lemma yields $U=U\cap X\F(U)=V\F(U)$ with $V\cap \F(U)=X\cap \F(U)=1$. From Lemma~\ref{aux} we know that $V=E$ thus $X\geq E$. Therefore,
$X/E$ complements $U/E=\Phi(G/E)$ in $G/E$ thus $X=G$ and so $\F(U)=1$. Finally, this implies that $U=E$ so $\Phi(G/E)$ is trivial. This final contradiction completes the proof of the theorem.
\end{proof}

Let us now demonstrate how Theorem~A can be used effectively to establish the generalisation of Doerk's theorem asserted in the introduction.

\begin{proof}[Proof of Theorem~B]
\ref{thmB1} First, we prove that $\B$ is a formation. Let $G$ be a group and suppose that $G/M, G/N \in \B$ for some normal subgroups $M,N$ of $G$ with $M \cap N =1$. Note that $\Phi(G)$ is 
contained in the full preimage of $\Phi(G/M)$, which is $M$, and in 
the full preimage of $\Phi(G/N)$, which is $N$. Thus 
$\Phi(G) \leq M \cap N =1$, and so $\Phi(G)=1$.
Now, let $G$ have least possible order among groups which 
have two trivially intersecting normal subgroups, each defining a 
quotient which is in $\B$, but the whole group is not in $\B$. We 
assume, as we may, that $M,N$ are minimal subgroups of $G$ 
with respect to $G/M, G/N \in \B$.
We claim that $M,N$ are both minimal normal subgroups of $G$. If not, then 
$1<U<M$ with $U \mathrel{\lhd} G$, for instance. By the minimality 
of $M$, $G/U \notin \B$. Consider the group $\overline{G}=G/U$ 
and its normal subgroups $\overline{M}=M/U$, $\overline{N}=UN/U$. 
Notice that $\overline{M} \cap \overline{N} = 1$ and both 
$\overline{G}\big/\overline{M}$, $\overline{G}/\overline{N}$ are 
in $\B$ since $G/M$, $G/N$ are in $\B$. However, this violates the 
minimality of $G$ since $\left|\overline{G}\right|<|G|$ and 
$\overline{M}, \overline{N}$ satisfy the initial hypotheses.
Then $\Phi(G)=1$ implies $\Soc(G)=\Fprime(G)=\Fast(G)$, thus $M \leq \Soc(G)=\Fast(G)$. Since 
$G/M \in \B$, we have $\Phi(G/K)=1$ for all $M \leq K$, thus 
$\Phi(G/K)=1$ for all $\Fast(G) \leq K$. Now Theorem~A
yields $G \in \B$, a contradiction. Therefore, there exists no such $G$. 

Next, we show that $\B$ is subnormally closed. For that it clearly
suffices to prove that $\B$ is closed under taking normal subgroups.
A straightforward induction on the subnormal defect then yields that
$\B$ is subnormally closed. Let $G$ have least 
possible order among groups which are in $\B$, but at least 
one of their normal subgroups is not. Let $N$ be one of them, note 
that $\Phi(N)=1$, and consider $\Fast(N)$, which is normal in $G$ 
and nontrivial. Then $G/\Fast(N)$ is in $\B$, since $G$ is in $\B$, 
but $N/\Fast(N)$ is not in $\B$ (if it were then Theorem~A
would yield $N\in \B$). This, however, contradicts the minimality 
of $G$, so there is no such $G$ to begin with.

Finally, we demonstrate that $\B$ is totally non-saturated. Let $G$
be a group and consider the $\B$-residual $T$ of
$G$. Let $T/H$ be a minimal normal subgroup of $G/H$, where
$H < T$ and $H \mathrel{\lhd} G$. If $T/H$ is not contained in
$\Phi(G/H)$ then $T/H \cap \Phi(G/H) = 1$ by minimality of $T/H$.
But $\Phi(G/H/(T/H)$ is trivial since $\Phi(G/T)$ is trivial,
so $\Phi(G/H)$ is contained in $T/H$, the full preimage of
$\Phi(G/H/T/H)$, and thus $\Phi(G/H) = 1$. Therefore,
$T/H \leq \Soc(G/H)=\Fast(G/H)$ and so $G/H/\Fast(G/H) \in \B$.
Then, by Theorem~A, we have that $G/H \in \B$, which
contradicts the fact that $T$ is the $\B$-residual of $G$, so $T/H$
is contained in $\Phi(G/H)$, as wanted.

\ref{thmB2} Let $\FF$ be the class of groups $G$ such that $\Phi (G/N) = \Phi(G)N/N$ for all $N \mathrel{\lhd} G$. Note that every $\Phi$-free group in $\FF$ is a $\B$-group. Hence $\FF$ is contained in $E_{\Phi}\B$. 

Assume that $X \in E_{\Phi}\B$. Then $X/\Phi(X)$ belongs to $\B$. Let $Z$ be a normal subgroup of $X$. Then the full preimages of $\Phi(X/Z\Phi(X))$ and $\Phi(G/Z)$ 
are equal since the intersection of those maximal subgroups of $X$ that contain $Z$ is precisely the intersection of the maximal subgroups that contain $Z\Phi(X)$. Since  $X/Z\Phi(X) \in \B$, it follows that $Z\Phi(X)/Z = \Phi(G/Z)$ and $X \in \FF$. Therefore $\FF = E_{\Phi}\B$. In particular, $\FF$ is contained in $\mathfrak{N}\B$. Suppose that $\mathfrak{N}\B$ is not contained in $\FF$ and let $G \in \mathfrak{N}\B \setminus \FF$ a group of minimal order. Then $G/\F(G) \in \B$. If $G$ were $\Phi$-free, then $\Soc(G)=\Fprime(G)=\Fast(G)$ and so $G/\Fast(G)$ would belong to $\B$. By Theorem~A, $G$ would be a $\B$-group, contrary to our supposition. Therefore, $\Phi(G) \neq 1$. The minimal choice of $G$ implies that $G/\Phi(G) \in \FF$ and so $G/\Phi(G) \in \B$. This contradiction shows that $\mathfrak{N}\B$ is contained in $\FF$, and so $\FF = E_{\Phi}\B = \mathfrak{N}\B$. 

\ref{thmB3} According to \cite[IV, Example 3.14 b)]{dh}, $\FF$ is a saturated formation which is locally defined by the formation function $f$ given by $f(p) \B$ for all primes $p$. Suppose that $\mathfrak{H}$ is a saturated formation containing $\B$. Then $\FF = E_{\Phi}\B \subseteq E_{\Phi}\mathfrak{H} = \mathfrak{H}$. Consequently, $\FF$ is the smallest saturated formation containing $\B$. By \cite[IV, Proposition 3.4]{dh}, $\FF$ is closed under taking subnormal subgroups. 
\end{proof}

\begin{proof}[Proof of Theorem~\ref{th-1}]
If $G$ is an $\nc$-group,  then certainly $G$ splits over each member of its generalised Fitting series, since each said member is normal.
For the other direction, we induct on the generalised Fitting length $k$ of the $\Phi$-free group $G$, that is, the smallest integer $k$ such that $\F_k^{\ast} = G$. If $k=1$ then $G = \Fast(G) = \F(G)\times\Lay(G)$ and if $N$ is a normal subgroup of $G$, it follows that $N = \Fast(N) = \F(N)\times\Lay(N)$ by Lemma~\ref{aux}. Since $\Lay(N)$ is a normal subgroup of $G$ which is a direct product of non-abelian simple groups, it follows that $\Lay(N)$ is complemented in $\Lay(G)$. Moreover $\F(N)$ is a normal subgroup of $G$ which is complemented in $\F(G)$ by a normal subgroup of $G$ since $\F(G)$ is a direct product of semisimple modules. Therefore $G$ splits over $N$.

Assume now that $k > 1$. Let $G \in \B$ have length $k$, and suppose that $G$ splits over each member of its generalised Fitting series. Let $H$ be a complement for $\Fast(G)$ in $G$. Then $H \in \B$ has generalised Fitting length $k-1$, since $G/\Fast(G) \cong H$, and is therefore an $\nc$-group by the inductive hypothesis. Let 
$N \mathrel{\lhd} G$ and write $X=\Fast(G)N$. Then $X \mathrel{\lhd} G$ and $X\cap H$ has a complement in $H$, say $C$. Let $V$ be a 
complement for $\Fast(N)=\Fast(G) \cap N$ in $\Fast(G)$. By our earlier remark $V \mathrel{\lhd} G$ and we claim that $VC$ is a complement 
for $N$ in $G$. First, $X=NV\Fast(N)=NV$, thus $NVC=XC=(X \cap H)C\Fast(G)=H\Fast(G)=G$. On the other hand 
$|C|=|H:X \cap H|=|XH:X|=|G:X|$ and $|V|=|\Fast(G):\Fast(G)\cap N|=|X:N|$. Thus $|C||V|=|G:N|$ and $C \cap V \leq H \cap \Fast(G)=1$. We 
conclude that $|VC|=|G:N|$, thus $VC$ is a complement for $N$ in $G$, as claimed. The induction is now complete.
\end{proof}

\begin{proof}[Proof of Corollary~\ref{subnc}]
First, we show that if $G \in \nc$ and $N \mathrel{\lhd}G$ then $N \in \nc$. 
Since $G \in \nc \subseteq \B$ and $B$ is subnormally closed it follows 
that $N \in B$ and so by Theorem~\ref{th-1} $N \in \nc$ if and only if $N$ splits over $\F_i^{\ast}(N)$ 
for all $i$. But $\F_i^{\ast}(N)$ is characteristic in $N$ thus normal in $G$, so if $C$ is a complement to $\F_i^{\ast}(N)$ in $G$ then $N \cap C$ is a complement to $\F_i^{\ast}(N)$ in $N$ by a standard application of Dedekind's lemma. Since $N$ splits over $\F_i^{\ast}(N)$ for all $i$ and $N \in \B$ we deduce that $N \in \nc$.

For the general case of subnormal subgroups of an $\nc$-group $G$ we argue 
by induction on its order, the base case being vacuously true. Let $G \in \nc$ 
and $H \mathrel{\lhd}\mathrel{\lhd}G$ and proper. Then 
$H \mathrel{\lhd}\mathrel{\lhd}N$ for some proper 
normal subgroup $N$ of $G$ which, by the previous paragraph, is an $\nc$-group. 
Applying the inductive hypothesis to $N < G$ we conclude that $H \in \nc$. Since 
$H$ was arbitrary our induction is complete.
\end{proof}

\begin{ack}
Part of this work has appeared in the first author's Ph.D. thesis. During that time Martin I. Isaacs contributed a proof of the implication \ref{newdoerk3} $\Rightarrow$ \ref{newdoerk1} in Corollary~\ref{newdoerk}.
Both authors are grateful to him for this contribution.
\end{ack}
\bibliographystyle{amsalpha}
\bibliography{Bibliography}

\end{document}